\documentclass[12pt]{article}

\usepackage{fancyhdr}
\newtheorem{theorem}{Theorem}

\newtheorem{corollary}[theorem]{Corollary}

\newtheorem{lemma}[theorem]{Lemma}
\newtheorem{proposition}[theorem]{Proposition}

\newenvironment{proof}[1][Proof]{\noindent\textbf{#1.} }{\ \rule{0.5em}{0.5em}}

\begin{document}

\vspace*{10mm}

\begin{center}
\uppercase{\textbf{On K\"{o}nig-Egerv\'{a}ry Square-Stable Graphs}}
\end{center}

\bigskip

\begin{center}
\textsc{Vadim E. Levit and Eugen Mandrescu}
\end{center}

\bigskip

\textsc{Abstract.}
The \textit{stability number} of a graph $G$, denoted by $\alpha(G)$, is the
cardinality of a maximum stable set, and $\mu(G)$ is the cardinality of a
maximum matching in $G$. If $\alpha(G)+\mu(G)$ equals its order, then $G$ is a
K\"{o}nig-Egerv\'{a}ry graph.

In this paper we deal with \textit{square-stable graphs}, i.e., the graphs $G$
enjoying the equality $\alpha(G)=\alpha(G^{2})$, where $G^{2}$ denotes the
second power of $G$. In particular, we show that a K\"{o}nig-Egerv\'{a}ry
graph is square-stable if and only if it has a perfect matching consisting of
pendant edges, and in consequence, we deduce that well-covered trees are
exactly the square-stable trees.

\bigskip


\bigskip


$$
\textsc{1. Introduction }
$$

All the graphs considered in this paper are simple, i.e., are finite,
undirected, loopless and without multiple edges. For such a graph $G=(V,E)$ we
denote its vertex set by $V=V(G)$ and its edge set by $E=E(G).$ If $X\subset
V$, then $G[X]$ is the subgraph of $G$ spanned by $X$. By $G-W$ we mean the
subgraph $G[V-W]$, if $W\subset V(G)$.

A set of pairwise non-adjacent vertices is \textit{a stable set} of $G$. The
\textit{stability number }of $G$, denoted by $\alpha(G)$, is the size
of a maximum stable set in $G$.

Let $\Omega(G)$ stand for the set $\{S:S$
\textit{is a maximum stable set of} $G\}$ and $core(G)=\cap\{S:S\in
\Omega(G)\}$ \cite{LeviMan3}.

$\theta(G)$ is the \textit{clique covering number} of $G$, i.e., the minimum
number of cliques whose union covers $V(G)$.

Recall that
\[
i(G)=\min\{|S|:S\ is\ a\ maximal\ stable\ set\ in\ G\},
\]
and
\[
\gamma(G)=\min\{|D|:D\ is\ a\ minimal\ domination\ set\ in\ G\}.
\]

A \textit{matching} is a set of non-incident edges of $G$; a matching of
maximum cardinality $\mu(G)$ is a \textit{maximum matching}, and a
\textit{perfect matching} is a matching covering all the vertices of $G$.

$G$ is called a \textit{K\"{o}nig-Egerv\'{a}ry graph} provided $\alpha(G)+\mu
(G)=\left\vert V(G)\right\vert $ \cite{dem,ster}. Various properties of K\"{o}nig-Egerv\'{a}ry graphs are presented in \cite{korach,KoNgPeis,levm2,LevMan3,LevMan2006,LevMan2007,LevMan2009,lov,lovpl}.

According to a well-known result of K\"{o}nig \cite{koen}, and Egerv\'{a}ry \cite{eger}, any bipartite graph is a K\"{o}nig-Egerv\'{a}ry graph. This class includes non-bipartite graphs as well (see, for instance, the
graphs $H_{1}$ and $H_{2}$ in Figure \ref{fig111}).

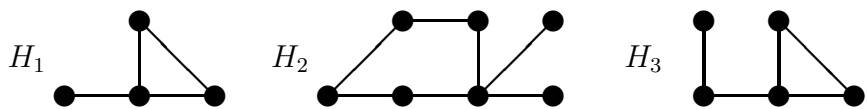
\begin{figure}[h]
\setlength{\unitlength}{1cm}\begin{picture}(5,1.3)\thicklines
\multiput(1.5,0)(1,0){3}{\circle*{0.29}}
\put(2.5,1){\circle*{0.29}}
\put(1.5,0){\line(1,0){2}}
\put(2.5,0){\line(0,1){1}}
\put(3.5,0){\line(-1,1){1}}
\put(1,0.5){\makebox(0,0){$H_{1}$}}
\multiput(5,0)(1,0){4}{\circle*{0.29}}
\put(7,1){\circle*{0.29}}
\put(6,1){\circle*{0.29}}
\put(8,1){\circle*{0.29}}
\put(5,0){\line(1,0){3}}
\put(6,1){\line(1,0){1}}
\put(7,0){\line(0,1){1}}
\put(5,0){\line(1,1){1}}
\put(7,0){\line(1,1){1}}
\put(4.5,0.5){\makebox(0,0){$H_{2}$}}
\multiput(10,0)(1,0){3}{\circle*{0.29}}
\multiput(10,1)(1,0){2}{\circle*{0.29}}
\put(10,0){\line(1,0){2}}
\put(10,0){\line(0,1){1}}
\put(11,0){\line(0,1){1}}
\put(11,1){\line(1,-1){1}}
\put(9.2,0.5){\makebox(0,0){$H_{3}$}}
\end{picture}\caption{Only $H_{3}$ is not a K\"{o}nig--Egerv\'{a}ry graph.}
\label{fig111}
\end{figure}

It is easy to see that if $G$ is a K\"{o}nig-Egerv\'{a}ry\emph{ }graph, then
$\alpha(G)\geq\mu(G)$, and that a graph $G$ having a perfect matching is a
K\"{o}nig-Egerv\'{a}ry\emph{ }graph if and only if $\alpha(G)=\mu(G)$.

The neighborhood of a vertex $v\in V$ is the set
\[
N(v)=\{w:w\in V,vw\in E\}, N[v]=N(v)\cup\{v\},
\]
and $N(A)=\cup\{N(v):v\in A\}$, for $A\subset V$.

If $G[N(v)]$ is a complete subgraph in $G$, then $v$ is a \textit{simplicial
vertex} of $G$, and by $simp(G)$ we mean the set of all simplicial vertices of
$G$. A maximal clique in $G$ is called a \textit{simplex} if it contains at
least a simplicial vertex of $G$ \cite{chhala}. $G$ is said to be
\textit{simplicial} if every vertex of $G$ is simplicial or it is adjacent to
a simplicial vertex of $G$ \cite{chhala}. If $\left\vert N(v)\right\vert=\left\vert \{w\}\right\vert =1$, then $v$ is a \textit{pendant vertex} and $vw$ is a \textit{pendant edge} of $G$.

By $K_{n}$ and $P_{n}$ we denote the complete graph and the cordless path on $n\geq1$ vertices, respectively. $K_{n,m}$ is the complete bipartite graph with two maximal stable sets of cardinalities $n$ and $m$.

$G$ is \textit{well-covered} \cite{plum} if every maximal
stable set of $G$ is also a maximum stable set, i.e., it is in $\Omega(G)$.
$G$ is called \textit{very well-covered} \cite{fav1} provided
$G$ is well-covered without isolated vertices, and $\left\vert
V(G)\right\vert =2\alpha(G)$.

\lhead{}
\chead{Levit \& Mandrescu - Square-Stable Graphs}
\rhead{}

The following results will be used in the sequel.

\begin{proposition}
\label{prop8}If $G$ is a graph of order $n\geq2$, then:

\begin{enumerate}
\item \cite{levm} $G$ is very well-covered if and only if it is a well-covered
K\"{o}nig-Egerv\'{a}ry graph;

\item \cite{levm} a connected K\"{o}nig-Egerv\'{a}ry graph is well-covered if
and only if it is very well-covered.

\item \cite{CampPlum} $H=G-N[v]$ is well-covered and $\alpha(H)=\alpha(G)-1$,
for every $v\in V(G)$, whenever $G$ is a non-complete well-covered graph.
\end{enumerate}
\end{proposition}

The distance between two vertices $v,w$ of a graph $G$ is denoted by
$dist_{G}(v,w)$, or $dist(v,w)$ if no ambiguity. $G^{2}$ denotes the
\textit{second power} of $G$, i.e., the graph with
\[
V(G^{2})=V(G),\quad E(G^{2})=E(G)\cup\{uv:u,v\in V(G),dist_{G}(u,v)=2\}.
\]
Clearly, any stable set of $G^{2}$ is stable in $G$, as well, while the
converse is not generally true. Therefore, we may assert that $1\leq
\alpha(G^{2})\leq\alpha(G)$. Notice that the both bounds are sharp. For
instance, if:

\begin{itemize}
\item $G=K_{1,n}$, and $n\geq2$, then $\alpha(G)=n>$ $\alpha(G^{2})=1$;

\item $G=P_{4}$, then $\alpha(G)=\alpha(G^{2})=2$.
\end{itemize}

\begin{figure}[h]
\setlength{\unitlength}{1cm} \begin{picture}(5,1.5)\thicklines
\multiput(1,0.5)(1,0){3}{\circle*{0.29}}
\multiput(4,0.5)(1,0){6}{\circle*{0.29}}
\multiput(10,0.5)(1,0){4}{\circle*{0.29}}
\multiput(1,1.5)(1,0){3}{\circle*{0.29}}
\multiput(4,1.5)(1,0){3}{\circle*{0.29}}
\multiput(8,1.5)(1,0){2}{\circle*{0.29}}
\multiput(10,1.5)(1,0){3}{\circle*{0.29}}
\put(2,0.5){\line(1,1){1}}
\put(1,0.5){\line(1,0){1}}
\put(2,1.5){\line(1,-1){1}}
\put(2,0.5){\line(1,0){1}}
\put(3,0.5){\line(0,1){1}}
\put(4,0.5){\line(1,0){5}}
\put(10,0.5){\line(1,0){3}}
\put(10,0.5){\line(0,1){1}}
\put(2,1.5){\line(1,0){1}}
\put(11,1.5){\line(1,0){1}}
\put(5,1.5){\line(1,-1){1}}
\put(5,0.5){\line(1,1){1}}
\put(5,1.5){\line(1,0){1}}
\put(6,0.5){\line(0,1){1}}
\put(7,0.5){\line(1,1){1}}
\multiput(1,0.5)(1,0){2}{\line(0,1){1}}
\multiput(4,0.5)(1,0){2}{\line(0,1){1}}
\multiput(8,0.5)(1,0){2}{\line(0,1){1}}
\multiput(11,0.5)(1,0){2}{\line(0,1){1}}
\put(1.5,0.1){\makebox(0,0){$G_{1}$}}
\put(6.5,0.1){\makebox(0,0){$G_{2}$}}
\put(11.5,0.1){\makebox(0,0){$G_{3}$}}
\end{picture}
\caption{$G_{1}$ and $G_{2}$ are square-stable graphs, while $G_{3}$ is not
square-stable.}
\end{figure}
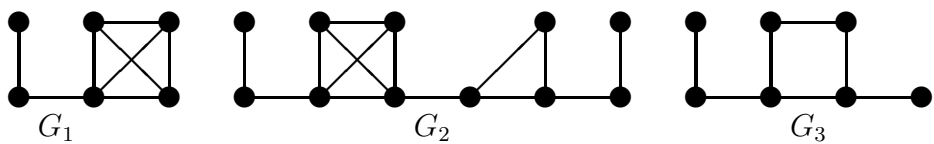

In this paper we characterize K\"{o}nig-Egerv\'{a}ry graphs satisfying the equality $\alpha(G)=\alpha(G^{2})$.
The graphs enjoying this property are called \textit{square-stable} \cite{LevManAlba}. The investigation of these graphs was started in \cite{ranvol} and continued in \cite{LevManAlba}.

\begin{center}
\textsc{2. Results}
\end{center}
It is evident that $G$ and $G^{2}$ are simultaneously connected or
disconnected. In addition, if $H_{i},1\leq i\leq k$, are the connected
components of graph $G$, then $S\in\Omega(G)$ if and only if $S\cap
V(H_{i})\in\Omega(H_{i}),1\leq i\leq k$. Hence, we may conclude that a
disconnected graph is square-stable if and only if each of its connected
components is square-stable. Therefore, in the rest of the paper all the
graphs are connected, unless otherwise stated.

In general, $\alpha(G^{2})\leq\theta(G^{2})\leq
\gamma(G)\leq i(G)\leq\alpha(G)\leq\theta(G)$.

\begin{theorem}
\label{th5}For a graph $G$ the following statements are equivalent:

\emph{(i)} every vertex of $G$ belongs to exactly one simplex of $G;$

\emph{(ii)} $G$ satisfies $\alpha(G)=\alpha(G^{2})$, i.e., $G$ is square-stable;

\emph{(iii)} $G$ satisfies $\theta(G)=\theta(G^{2})$;

\emph{(iv)} $G$ satisfies $\alpha(G^{2})=\theta(G^{2})=\gamma(G)=i(G)=\alpha
(G)=\theta(G)$;

\emph{(v)} $G$ is simplicial and well-covered;

\emph{(vi)} there exists some $S\in\Omega(G)$ such that $dist_{G}(a,b)\geq3$ holds
for all distinct $a,b\in S$.
\end{theorem}

Recall that: the equivalence of the assertions \emph{(i)}, \emph{(ii)},
\emph{(iii)}, \emph{(iv)}, was proved in \cite{ranvol}, the fact that
"\emph{(i)} $\iff$ \emph{(v)}" was\emph{ }shown in \cite{pristopvest}, while
"\emph{(ii)} $\iff$ \emph{(vi)}" was done in \cite{LevManAlba}.

For square-stable graphs there are interesting connections between properties
of vertices in $G$ and $G^{2}$ correspondingly.

\begin{proposition}
\label{prop11}If $G$ is a square-stable graph and $v\in V(G)$, then the
following assertions are true:

\emph{(i)} $v$ is in some $S\in\Omega(G^{2})$ if and only if $v$ is a
simplicial vertex in $G$;

\emph{(ii)} $v$ belongs to all maximum stable sets of $G^{2}$ if and
only if $v$ is the unique simplicial vertex in its simplex in $G$.
\end{proposition}

\begin{proof}
\emph{(i)} Let $v$ belong to some $S\in\Omega(G^{2})$. If $v$ is not
simplicial in $G$, then there are $x,y\in N_{G}(v)$, such that $xy\notin E(G)$.
In addition,
\[
N_{G}(x)\cap S=\{v\}=N_{G}(y)\cap S,
\]
because $dist_{G}(v,u)\geq3$, for each $u\in S-\{v\}$.
Since $S$ belongs also to $\Omega(G)$, we infer that $(S\cup\{x,y\})-\{v\}$
must be stable in $G$ and of cardinality greater than $\left\vert S\right\vert $,
thus contradicting $\left\vert S\right\vert =\alpha(G)$.

Conversely, let $v$ be a simplicial vertex in $G$. Suppose that $S\in
\Omega(G^{2})$, and let $\{u\}=S\cap N_{G}[v]$. Hence, $dist_{G}(x,v)\geq3$
for all $x\in S-\{u\}$. Otherwise, there exists a vertex $x_{0}\in S-\{u\}$
such that $dist_{G}(v,x_{0})=2$, which implies $dist_{G}(u,x_{0})=2$, because
$v$ is simplicial in $G$. Therefore,
\[
S^{\prime}=(S-\{u\})\cup\{v\}\in\Omega(G^{2})\ and\ v\in S^{\prime}.
\]
\emph{(ii)} Suppose that a vertex $v$ belongs to all maximum stable
sets of $G^{2}$. According to \emph{(i)}, $v$ must be a simplicial
vertex in $G$. If some $u\in N_{G}(v)$ is also simplicial in $G$, then again
by \emph{(i)}, $u$ belongs to some $S\in\Omega(G^{2})$. Clearly, $v\notin S$
and this contradicts the hypothesis on $v$.

Conversely, let $v$ be the unique simplicial vertex in its simplex. Assume
that there is some $S\in\Omega(G^{2})$ such that $v\notin S$. Since $G$ is
square-stable, $S\in\Omega(G)$. Clearly, there exists $u\in S\cap N_{G}(v)$,
because, otherwise, $S\cup\{v\}$ is stable in $G$ and $\left\vert
S\cup\{v\}\right\vert >\alpha(G)$. Hence, $S\cap N_{G}(u)=\emptyset$, and
each $x\in S$ satisfies $dist_{G}(x,u)\geq2$.

The vertex $u$ is not simplicial, since $v$ is the unique simplicial vertex of
its simplex and $u\in N_{G}(v)$. Consequently, there is $y\in N_{G}%
(u)-N_{G}\left[  v\right]  $.

Now we get that $(S\cup\{v,y\})-\{u\}$ is stable
in $G$, because $S\in\Omega(G^{2})$ means that $dist_{G}(x,u)\geq3$, which
implies $dist_{G}(x,y)\geq2$ for all $x\in S-\{u\}$. Thus, we get that
\[
\left\vert (S\cup\{v,y\})-\{u\}\right\vert >\alpha(G),
\]
in contradiction with the definition of $\alpha(G)$. In summary, $v$ must
belong to all maximum stable sets of $G^{2}$.
\end{proof}

The graph in Figure \ref{1515} illustrates interconnections between $simp(G)$ and $core(G^{2})$ implied by Proposition \ref{prop11}.

\begin{figure}[h]
\setlength{\unitlength}{1cm} \begin{picture}(5,2.3)\thicklines
\multiput(3,0.5)(1,0){4}{\circle*{0.29}}
\multiput(3,1.5)(1,0){4}{\circle*{0.29}}
\put(4,1.5){\line(1,0){1}}
\put(5,0.5){\line(1,0){1}}
\multiput(3,0.5)(1,0){4}{\line(0,1){1}}
\multiput(3,0.5)(1,0){3}{\line(1,1){1}}
\put(2.3,1){\makebox(0,0){$G$}}
\put(3,0){\makebox(0,0){$v_{1}$}}
\put(4,0){\makebox(0,0){$v_{2}$}}
\put(5,0){\makebox(0,0){$v_{3}$}}
\put(6,0){\makebox(0,0){$v_{4}$}}
\put(3,2){\makebox(0,0){$v_{5}$}}
\put(4,2){\makebox(0,0){$v_{6}$}}
\put(5,2){\makebox(0,0){$v_{7}$}}
\put(6,2){\makebox(0,0){$v_{8}$}}
\multiput(8,0.5)(1,0){4}{\circle*{0.29}}
\multiput(8,1.5)(1,0){4}{\circle*{0.29}}
\multiput(8,0.5)(1,0){4}{\line(0,1){1}}
\multiput(8,0.5)(1,0){3}{\line(1,1){1}}
\multiput(9,1.5)(1,0){2}{\line(1,-1){1}}
\put(8,1.5){\line(1,0){3}}
\put(8,0.5){\line(1,0){3}}
\put(8,0.5){\line(2,1){2}}
\put(7.3,1){\makebox(0,0){$G^{2}$}}
\put(8,0){\makebox(0,0){$v_{1}$}}
\put(9,0){\makebox(0,0){$v_{2}$}}
\put(10,0){\makebox(0,0){$v_{3}$}}
\put(11,0){\makebox(0,0){$v_{4}$}}
\put(8,2){\makebox(0,0){$v_{5}$}}
\put(9,2){\makebox(0,0){$v_{6}$}}
\put(10,2){\makebox(0,0){$v_{7}$}}
\put(11,2){\makebox(0,0){$v_{8}$}}
\end{picture}
\caption{$G$ is square-stable and $v_{2},v_{4},v_{5},v_{8}\in simp(G)$, but
only $v_{2},v_{5}\in core(G^{2})$.}
\label{1515}
\end{figure}
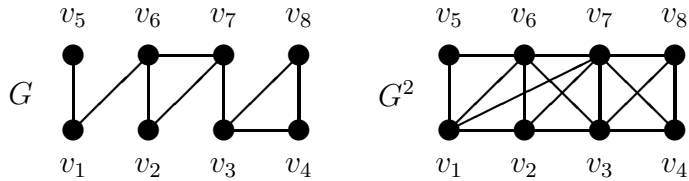

Clearly, any complete graph is square-stable. Moreover, since $K_{n}^{2}%
=K_{n}$, we get that
\[
\left\vert \Omega(K_{n})\right\vert =\left\vert
\Omega(K_{n}^{2})\right\vert =\left\vert \{\{v\}:v\in V(K_{n})\}\right\vert
=n.
\]
Let us notice that the equality $\left\vert
\Omega(G^{2})\right\vert =1$ does not ensure that $G$ is square-stable, e.g.,
the graph $G$ from Figure \ref{565656} is not square-stable, but $\left\vert
\Omega(G^{2})\right\vert =1$.

\begin{figure}[h]
\setlength{\unitlength}{1cm} \begin{picture}(5,2.3)\thicklines
\multiput(3,1)(1,0){2}{\circle*{0.29}}
\multiput(3,2)(2,0){2}{\circle*{0.29}}
\put(4,0){\circle*{0.29}}
\put(6,2){\circle*{0.29}}
\put(3,1){\line(1,0){1}}
\put(3,1){\line(0,1){1}}
\put(3,1){\line(1,-1){1}}
\put(5,2){\line(1,0){1}}
\put(3,1){\line(2,1){2}}
\put(3,2){\line(1,-1){1}}
\put(4,1){\line(1,1){1}}
\put(4,0){\line(1,2){1}}
\put(2.3,1){\makebox(0,0){$G$}}
\put(2.7,2){\makebox(0,0){$x$}}
\put(6.3,2){\makebox(0,0){$y$}}
\put(4.3,0){\makebox(0,0){$z$}}
\multiput(8,1)(1,0){2}{\circle*{0.29}}
\multiput(8,2)(2,0){2}{\circle*{0.29}}
\put(9,0){\circle*{0.29}}
\put(11,2){\circle*{0.29}}
\put(8,1){\line(1,0){1}}
\put(8,1){\line(0,1){1}}
\put(8,1){\line(1,-1){1}}
\put(8,2){\line(1,0){2}}
\put(8,2){\line(1,-2){1}}
\put(10,2){\line(1,0){1}}
\put(8,1){\line(2,1){2}}
\put(8,2){\line(1,-1){1}}
\put(9,1){\line(1,1){1}}
\put(9,0){\line(0,1){1}}
\put(9,0){\line(1,2){1}}
\put(8,1){\line(3,1){3}}
\put(9,1){\line(2,1){2}}
\put(9,0){\line(1,1){2}}
\put(7.3,1){\makebox(0,0){$G^{2}$}}
\put(7.7,2){\makebox(0,0){$x$}}
\put(11.3,2){\makebox(0,0){$y$}}
\end{picture}
\caption{$\alpha(G)=\left\vert \{x,y,z\}\right\vert =3$, while $\Omega
(G^{2})=\{\{x,y\}\}$.}
\label{565656}
\end{figure}
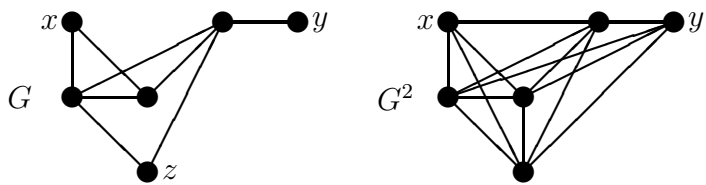

In general, a graph having a unique perfect matching is not necessarily
square-stable. For instance, $K_{3}+e$ has a unique perfect matching, but is
not square-stable, because $(K_{3}+e)^2=K_4$.

Further, we pay attention to graphs having a perfect matching
consisting of pendant edges, which is obviously unique.

\begin{lemma}
\label{prop6}If $G$ has a perfect matching consisting of pendant edges, then
$G$ is square-stable. Moreover, $\Omega(G^{2})=\{S_{0}\}$, where $S_{0}=\{v:v$
\textit{is pendant in }$G\}$.
\end{lemma}

\begin{proof}
Clearly, the set $S_{0}$ is stable in $G$, and because
\[
\left\vert V(G)\right\vert /2=\mu(G)=\left\vert S_{0}\right\vert \leq
\alpha(G)\leq\left\vert V(G)\right\vert -\mu(G)=\left\vert V(G)\right\vert
/2,
\]
we infer that $S_{0}\in\Omega(G)$. Since $dist_{G}(a,b)\geq3$ holds for any
$a,b\in S_{0}$, we get that $S_{0}\in\Omega(G^{2})$, i.e., $G$ is
square-stable, by Theorem \ref{th5}\emph{(vi)}.

Since a pendant vertex is the unique simplicial vertex in its simplex,
Proposition \ref{prop11} implies that $S_{0}$ is included in all maximum
stable sets of $G^{2}$. So, we may conclude that $\Omega(G^{2})=\{S_{0}\}$,
because $S_{0}\in\Omega(G^{2})$.
\end{proof}

Notice that there are square-stable graphs with more than one maximum
stable set, and having no perfect matching; e.g., the graph in Figure
\ref{fig111111}.

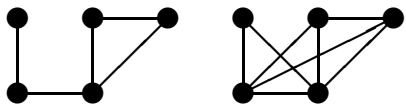
\begin{figure}[h]
\setlength{\unitlength}{1cm}\begin{picture}(5,1.2)\thicklines
\multiput(4,0)(1,0){2}{\circle*{0.29}}
\multiput(4,1)(1,0){3}{\circle*{0.29}}
\put(4,0){\line(1,0){1}}
\put(5,1){\line(1,0){1}}
\put(5,0){\line(0,1){1}}
\put(5,0){\line(1,1){1}}
\put(4,0){\line(0,1){1}}
\multiput(7,0)(1,0){2}{\circle*{0.29}}
\multiput(7,1)(1,0){3}{\circle*{0.29}}
\put(7,0){\line(1,0){1}}
\put(8,1){\line(1,0){1}}
\put(7,0){\line(1,1){1}}
\put(7,0){\line(0,1){1}}
\put(7,1){\line(1,-1){1}}
\put(8,0){\line(0,1){1}}
\put(8,0){\line(1,1){1}}
\put(7,0){\line(2,1){2}}
\end{picture}
\caption{A square-stable graph $G$ and its $G^{2}$. $G$ is not very
well-covered.}
\label{fig111111}
\end{figure}

\begin{theorem}
\label{prop9}For a K\"{o}nig-Egerv\'{a}ry graph $G$ on $n\geq2$ vertices the
following assertions are equivalent:

\emph{(i)} $G$ square-stable;

\emph{(ii)} $G$ has a perfect matching consisting of pendant edges;

\emph{(iii)} $G$ is very well-covered with exactly $\alpha(G)$ pendant vertices.
\end{theorem}

\begin{proof}
\emph{(i)} $\Rightarrow$ \emph{(ii)} By Theorem \ref{th5}\emph{(v)}, $G$ is
well-covered, and according to Proposition \ref{prop8}\emph{(2)}, it is also
very well-covered. Hence, we obtain that $\alpha(G)=\mu(G)=n/2$, and $G$ has a
perfect matching $M$. Let
\[
S_{0}=\{a_{i}:1\leq i\leq\alpha(G)\}\in\Omega(G^{2})
\]
and $b_{i}\in V(G)-S_{0}$ be such that $a_{i}b_{i}\in M$ for $1\leq i\leq\alpha(G)$. The
inequality $dist_{G}(v,w)\geq3$ holds for any $v,w\in S_{0}$, because
otherwise $vw\in E(G^{2})$, and this contradicts the stability of $S_{0}$ in
$G^{2}$.

If $b_{j}\in N_{G}(a_{i})$ for some $i\neq j$, it follows that
$dist_{G}(a_{i},a_{j})=2$, in contradiction with $dist_{G}(a_{i},a_{j})\geq3$.
Therefore, $N_{G}(a_{i})=\{b_{i}\},1\leq i\leq\alpha(G)$, i.e., $M$ consists of only pendant edges.

\emph{(ii)} $\Rightarrow$ \emph{(iii)} By Lemma \ref{prop6}, it follows that
$G$ is square-stable. According to Theorem \ref{th5}\emph{(v)}, $G$ is
well-covered, and Proposition \ref{prop8}\emph{(2)} finally assures that $G$
is very well-covered.

\emph{(iii)} $\Rightarrow$ \emph{(i)} Since $G$ is very well-covered with
exactly $\alpha(G)$ pendant vertices, we infer that
\[
S_{0}=\{v:v \ is\ a\ pendant\ vertex\ in\ G\}\in\Omega(G)
\]
and also that the matching $M=\{vw:vw\in
E(G),v\in S_{0}\}$ is perfect and consists of only pendant edges. According to
Lemma \ref{prop6}, it follows that $G$ is square-stable.
\end{proof}

Let us remark that:

\begin{itemize}
\item well-covered K\"{o}nig-Egerv\'{a}ry graphs may not be
square-stable, e.g., $C_{4}$;

\item a K\"{o}nig-Egerv\'{a}ry graph with a unique perfect matching is not
always square-stable, e.g., $P_{6}$ (by the way, it is also a tree) and
$K_{3}+e$;

\item a non-K\"{o}nig-Egerv\'{a}ry graph with a unique perfect matching $M$
may be square-stable, even if $M$ does not consist of only pendant edges (for
instance, see the graph in Figure \ref{fig12222}).
\end{itemize}

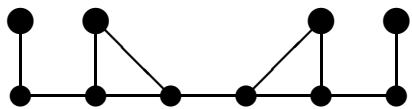
\begin{figure}[h]
\setlength{\unitlength}{1.0cm} \begin{picture}(5,1)\thicklines
\multiput(4,0)(1,0){6}{\circle*{0.29}}
\multiput(4,1)(1,0){2}{\circle*{0.35}}
\multiput(8,1)(1,0){2}{\circle*{0.35}}
\put(4,0){\line(1,0){5}}
\put(5,1){\line(1,-1){1}}
\put(7,0){\line(1,1){1}}
\multiput(4,0)(1,0){2}{\line(0,1){1}}
\multiput(8,0)(1,0){2}{\line(0,1){1}}
\end{picture}
\caption{$G$ is square-stable and has a unique perfect matching containing not
only pendant edges.}
\label{fig12222}
\end{figure}

Theorem \ref{prop9} is true for bipartite graphs as well, since any bipartite
graph is also a K\"{o}nig-Egerv\'{a}ry graph. For trees, Theorem \ref{prop9}
leads to an extension of the characterization that Ravindra gave to
well-covered trees in \cite{rav1}.

\begin{corollary}
\label{cor1}If $T$ is a tree of order $n\geq2$, then the following statements
are equivalent:

\emph{(i)} $T$ is well-covered;

\emph{(ii)} $T$ is very well-covered;

\emph{(iii)} $T$ has a perfect matching consisting of pendant edges;

\emph{(iv)} $T$ is square-stable.
\end{corollary}

Let us notice that the equivalences appearing in Corollary \ref{cor1} fail for
bipartite graphs. For instance, the graph in Figure \ref{fig4}\emph{\ }is very
well-covered, but is not square-stable.

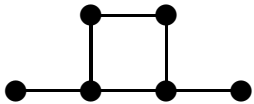
\begin{figure}[h]
\setlength{\unitlength}{1cm}\begin{picture}(5,1)\thicklines
\multiput(5,0)(1,0){4}{\circle*{0.29}}
\multiput(6,1)(1,0){2}{\circle*{0.29}}
\put(5,0){\line(1,0){3}}
\put(6,0){\line(0,1){1}}
\put(7,0){\line(0,1){1}}
\put(6,1){\line(1,0){1}}
\end{picture}
\caption{A very well-covered bipartite graph, which is not square-stable.}
\label{fig4}
\end{figure}

\begin{lemma}
\label{Lemma1} If $G$ is square-stable with $2$ vertices at least,
then $\alpha(G)\leq\mu(G)$.
\end{lemma}

\begin{proof}
By Theorem \ref{th5}\emph{(vi)} there exists a maximum stable set
\[
S=\left\{  v_{i}:1\leq i\leq\alpha(G)\right\}
\]
in $G$ such that $dist_{G}\left(  a,b\right)  \geq3$ for all $a,b\in S$. It
means that for every $i\in\left\{  1,2,...,\alpha\left(  G\right)  -1\right\}
$ there is a shortest path of length $3$ at least from $v_{i}$ to
$v_{\alpha\left(  G\right)  }$: $v_{i},w_{i},...w^{i},$ $v_{\alpha\left(
G\right)  }$. All the vertices $w_{i},1\leq i\leq\alpha\left(  G\right)  -1$
and $w^{1}$ are pairwise distinct, otherwise there will be a pair of vertices
in $S$ at distance $2$. Hence
\[
M=\left\{  v_{i}w_{i}:1\leq i\leq\alpha(G)-1\right\}  \cup\left\{
v_{\alpha(G)}w^{1}\right\}
\]
is a matching in $G$. Thus $\alpha(G)=\left\vert S\right\vert \leq\left\vert
M\right\vert \leq\mu(G)$.
\end{proof}

\begin{proposition}
\label{Cor10} Let $G^{2}$ be  a K\"{o}nig-Egerv\'{a}ry graph with $2$ vertices
at least. Then $G$ is square-stable if and only if $G$ is a
K\"{o}nig-Egerv\'{a}ry graph with a perfect matching.
\end{proposition}

\begin{proof}
The following inequalities are true for every graph $G$:
\[
\mu(G)\leq\mu(G^{2})\ and\ \alpha(G^{2})\leq\alpha(G).
\]

Since $G^{2}$ is a K\"{o}nig-Egerv\'{a}ry graph, $\mu(G^{2})\leq\alpha\left(
G^{2}\right)  $. Consequently,
\[
\mu(G)\leq\mu(G^{2})\leq\alpha(G^{2})\leq\alpha(G).
\]

If $G$ is square-stable, then these inequalities together with Lemma
\ref{Lemma1}\ give
\[
\mu(G)=\mu(G^{2})=\alpha(G^{2})=\alpha(G).
\]
Moreover,
\[
\left\vert V(G)\right\vert =\mu(G^{2})+\alpha(G^{2})=\mu(G)+\alpha(G),
\]
which means that $G$ is a K\"{o}nig-Egerv\'{a}ry graph. In addition, $G$\ has
a perfect matching, because $\mu(G)=\alpha\left(  G\right)  $.

Conversely, if $G$ is a K\"{o}nig-Egerv\'{a}ry graph with a perfect matching,
then
\[
\mu(G)+\alpha(G)=\left\vert V(G)\right\vert =\mu(G^{2})+\alpha(G^{2})\
and\ \mu(G)=\mu(G^{2}).
\]
Thus $\alpha(G)=\alpha(G^{2})$, i.e., $G$ is a square-stable graph.
\end{proof}

It is worth noticing that if $G$ is square-stable, then it is not enough to know that
$\mu(G)=\alpha(G)$ in order to be sure that $G$ is a
K\"{o}nig-Egerv\'{a}ry graph. For instance, see Figure \ref{fig12}.

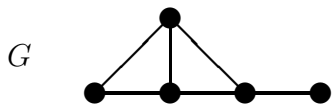
\begin{figure}[h]
\setlength{\unitlength}{1cm}\begin{picture}(5,1)\thicklines
\multiput(6,0)(1,0){4}{\circle*{0.29}}
\put(7,1){\circle*{0.29}}
\put(6,0){\line(1,0){3}}
\put(7,0){\line(0,1){1}}
\put(8,0){\line(-1,1){1}}
\put(6,0){\line(1,1){1}}
\put(5,0.5){\makebox(0,0){$G$}}
\end{picture}\caption{$G$ is a square-stable non-K\"{o}nig-Egerv\'{a}ry graph: $\mu(G)=\alpha(G)=2$.}
\label{fig12}
\end{figure}

Combining Propositions \ref{prop8}, \ref{Cor10} and Theorem \ref{prop9}, we
obtain the following.

\begin{corollary}
\emph{(i)} $G$ is square-stable and very well-covered if and only if $G$ is a
K\"{o}nig-Egerv\'{a}ry graph with exactly $\alpha(G)$ pendant vertices.

\emph{(ii)} If $G$ is square-stable, then either $G$\ and $G^{2}$ are
K\"{o}nig-Egerv\'{a}ry graphs or both of them are not.
\end{corollary}

Another consequence of Theorem \ref{prop9} is the following extension of the
characterization that Finbow \textit{et al.} gave in \cite{finhart1} for
well-covered graphs having the girth $\geq6$.

\begin{theorem}
Let $G$ be a graph of girth $\geq6$, which is isomorphic to neither $C_{7}$
nor $K_{1}$. Then the following assertions are equivalent:

\emph{(i)} $G$ is well-covered;

\emph{(ii)} $G$ has a perfect matching consisting of pendant edges;

\emph{(iii)} $G$ is very well-covered;

\emph{(iv)} $G$ is a K\"{o}nig-Egerv\'{a}ry graph with exactly $\alpha(G)$ pendant vertices, and $core(G)=\emptyset$;

\emph{(v)} $G$ is a K\"{o}nig-Egerv\'{a}ry square-stable graph.
\end{theorem}

\begin{proof}
The equivalences \emph{(i)} $\Leftrightarrow$ \emph{(ii)} $\Leftrightarrow$
\emph{(iii)} are proved in \cite{finhart1}. In \cite{levm} it is shown
that \emph{(iii)} $\Leftrightarrow$ \emph{(iv)}. Finally, \emph{(ii)}
$\Leftrightarrow$ \emph{(v)} is true according to Lemma \ref{prop6} and Theorem
\ref{prop9}.
\end{proof}

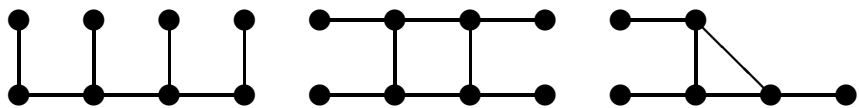
\begin{figure}[h]
\setlength{\unitlength}{1.0cm} \begin{picture}(5,1.2)\thicklines
\multiput(1,0)(1,0){12}{\circle*{0.29}}
\multiput(1,1)(1,0){10}{\circle*{0.29}}
\put(1,0){\line(1,0){3}}
\put(5,0){\line(1,0){3}}
\put(5,1){\line(1,0){3}}
\put(9,0){\line(1,0){3}}
\put(9,1){\line(1,0){1}}
\multiput(1,0)(1,0){4}{\line(0,1){1}}
\multiput(6,0)(1,0){2}{\line(0,1){1}}
\put(10,0){\line(0,1){1}}
\put(10,1){\line(1,-1){1}}
\end{picture}
\caption{Square-stable K\"{o}nig-Egerv\'{a}ry graphs.}
\label{fig1}
\end{figure}

\begin{center}
\textsc{3. Conclusions}
\end{center}

In this paper we concentrated on K\"{o}nig-Egerv\'{a}ry graphs enjoying the
equality $\alpha(G^{2})=\alpha(G)$. It seems to be interesting to study other
families of graphs satisfying the same property.



\bigskip

\noindent
Vadim E. Levit \\
Department of Mathematics and Computer Science\\
Ariel University Center of Samaria\\
Ariel 40700, Israel\\
email:\textit{levitv@ariel.ac.il}

\bigskip

\noindent
Eugen Mandrescu \\
Department of Computer Science\\
Holon Institute of Technology\\
52 Golomb Str., Holon, Israel\\
email:\textit{eugen\_m@hit.ac.il}
\end{document}